\documentclass[psamsfonts]{amsproc}

\usepackage{mathtools}
\usepackage{dsfont}
\usepackage{upref}
\usepackage{amssymb}
\usepackage[all]{xy}
\newtheorem{theorem}{Theorem}[section]
\newtheorem{lemma}[theorem]{Lemma}
\newtheorem{proposition}[theorem]{Proposition}
\newtheorem{definition-proposition}[theorem]{Definition-Proposition}

\theoremstyle{definition}
\newtheorem{definition}[theorem]{Definition}
\newtheorem{example}[theorem]{Example}

\theoremstyle{remark}

\numberwithin{equation}{section}

\newcommand{\QQ}{\mathcal Q}
\newcommand{\PP}{\mathcal P}
\newcommand{\XX}{\mathcal X}

\newcommand{\cO}{\mathcal{O}}

\begin{document}

\title{
The Super Lie Groups Associated to Odd Involutions\footnote{2010 Mathematics Subject Classification. Primary 58A50; Secondary 20N99. }}

\author[Mohammadi M.]{Mohammad Mohammadi}
\address{Department of Mathematics, Institute for Advanced Studies in Basic Sciences (IASBS), No. 444, Prof. Yousef Sobouti Blvd.
P. O. Box 45195-1159 Zanjan Iran, Postal Code 45137-66731}
\curraddr{}
\email{moh.mohamady@iasbs.ac.ir}
\thanks{}

\author[Varsaie S.]{Saad Varsaie}
\address{Department of Mathematics, Institute for Advanced Studies in Basic Sciences (IASBS), No. 444, Prof. Yousef Sobouti Blvd.
P. O. Box 45195-1159 Zanjan Iran, Postal Code 45137-66731}
\curraddr{}
\email{varsaie@iasbs.ac.ir}
\thanks{}

\date{}

\begin{abstract}
A new generalization of Grassmannians in supergeometry, called $\nu-$Grassmannians, are constructed by gluing $\nu-$domains. By a $\nu-$domain, we mean a superdomain with an odd involution say $\nu$ on its structure sheaf, as morphism of modules. Then we show that $\nu-$Grassmannians are homogeneous superspaces. In addition, in the last section, a supergroup associated to the odd involution $\nu$ is introduced.

\medskip

\textbf{Keywords:} Super Lie group, supergrassmannian, homogeneous super space, $\nu-$Grassmannian.
\end{abstract}

\maketitle

\section*{Introduction}
In this paper, we introduced a super Lie group in section \ref{newgroup}.  
As far as we know, such a super Lie group has not previously been reported in existing resources. Every thing start from our attempt to generalizing the concept of Chern classes in supergeometry. For this, the classifying space approach was considered. One can show that the supergrassmannians are not suitable candidate as classifying spaces. See \cite{Afshari} and \cite{Roshandel}.

For this reason, we introduced a new generalization of Grassmannians in supergeometry called $\nu-$Grassmannians. These spaces are constructed by using an odd involution on superdomains \cite{OurArticle}. It is shown that the canonical super vector bundle over $\nu-$Grassmannians are universal. In addition, there are cohomology elements, associated to these bundles, which play role as universal Chern classes . See \cite{Afshari} and \cite{Roshandel} for more details. 

Since these spaces are constructed by gluing superdomains, it is good to have a global description for $\nu-$Grassmannians. In order to solve this problem, we show $\nu-$Grassmannians are homogeneous superspaces. While studying the homogeneity of $\nu-$Grassmannians, we find a new supergroup associated to the odd involution $\nu$. One may see more details in the last section. In \cite{Afshari}, \cite{OurArticle}, \cite{MyArticle} and \cite{Roshandel}, different properties of $\nu$- Grassmannians are studied. These works may be considered as the initial steps in introducing $\nu-$supergeometry.

In the first section, we begin by elementary category theory and studying supermanifolds as ringed spaces with particular conditions. We define the category of super Lie groups and super Harish-Chandra pairs as two equivalent categories. Also, we study the actions of super Lie groups on supermanifolds. 

In the second section, we explain the construction of $\nu-$Grassmannians by gluing superdomains. In the third section, we will show that the super Lie group $GL(m|n)$ acts on the supermanifold $_{\nu}G_{k|l}(m|n)$ transitively and it follows that $_{\nu}G_{k|l}(m|n)$ are homogeneous superspaces. In the last section, we will introduce a super Lie subgroup of $GL(m|n)$ associated the involution $\nu.$

\section{Preliminaries} \label{prelim}
In this section, we introduce the basic definitions and results concerning category theory, supermanifolds, $\nu$-manifolds, super Lie groups and action of a super Lie group on a supermanifold and super Harish-Chandra pairs. See \cite{carmelibook} and \cite{vmanifold} for further details.

\subsection{Category theory}
Let $\mathcal{C}$ be a locally small category, and $X$ be an object in $\mathcal{C}$. For any $T\in Obj(\mathcal{C})$, the set $X(T):=Hom_\mathcal{C}(T,X)$ called $T$-points of $X$. Let $\textbf{SET}$ be the category of sets and consider the following functor
$$\begin{matrix} X(.):\mathcal{C} \rightarrow \textbf{SET}\\
\qquad\quad S \mapsto X(S),\\
\\
X(.): Hom_{\mathcal{C}}(S,T)\rightarrow Hom_{\textbf{SET}}(X(T),X(S))\\
\varphi \mapsto X(\varphi),\end{matrix}$$\\
where $X(\varphi):f\mapsto f\circ \varphi.$ The above functor called \textbf{functor of points of $X$}. 

\medskip

Corresponding to each morphism $\psi:X\rightarrow Y$ in $\mathcal C$, there exists a natural transformation
$\psi(.)$ from $X(.)$ to $Y(.)$. So for each object $T,$ one can define the map 
$\psi(T):X(T)\rightarrow Y(T)$ with $\xi\mapsto \psi\circ \xi$ . Now set:
$$\begin{matrix} \mathcal{Y}:\mathcal{C}\rightarrow [\mathcal C,\textbf{SET}]\\ X \mapsto X(.)\\ \psi \mapsto \psi(.). \end{matrix}$$
where $[\mathcal{C}, \textbf{SET}]$ is the category of functors from $\mathcal{C}$ to $\textbf{SET}$. Obviously, $\mathcal{Y}$ is a covariant  functor and it is called \textit{\textbf{Yoneda embedding}}.
Yoneda lemma says that Yoneda embedding is full and faithful functor, i.e. the map $$Hom_\mathcal{C}(X,Y)\longrightarrow Hom_{[\mathcal{C},\textbf{SET}]}(X(.),Y(.)),$$ 
is a bijection for each $X,Y\in Obj(\mathcal{C}). $ 
%
According to Yoneda lemma, $X,Y\in Obj(\mathcal{C})$ are isomorphic if and only if their functor of points are isomorphic. Yoneda embedding is an equivalence between $\mathcal{C}$ and a subcategory of representable functors in $[\mathcal{C},\textbf{SET}]$.

\medskip

Let $X$, $Y$ are objects in a category and $\alpha,\beta:X\rightarrow Y$ are morphisms between these objects. An universal pair $(E,\epsilon)$ is called equalizer if the following diagram commutes:
$$E\xrightarrow{\epsilon} X\overset{\alpha}{\underset{\beta}{\rightrightarrows}}Y$$
i.e., $\alpha \circ \epsilon=\beta \circ \epsilon$ and also for each object $T$ and any morphism $\tau:T\rightarrow X$ which satisfy $\alpha \circ \tau=\beta \circ \tau$, there exists unique morphism $\sigma:T\rightarrow E$ such that $\epsilon \circ \sigma = \tau$. If equalizer existed then it is unique up to isomorphism. For example, in the category of sets, which is denoted by $\textbf{SET}$, the equalizer of two morphisms $\alpha,\beta:X\rightarrow Y$ is the set $E=\{x\in X | \alpha(x)=\beta(x)\}$ together with the inclusion map  $\epsilon:E\hookrightarrow X$
\medskip
\subsection{supermanifolds}\label{subsection2}
By a super ringed space, we mean a pair $(X, \cO_X)$ where $X$ is a topological space and $\mathcal O_X$ is a sheaf of supercommutative $\mathbb Z_2$-
graded rings on X. A morphism between 
$(X, \cO_X)$ and
$(X, \cO_Y)$ is a pair $\psi:=(\overline {\psi},\psi^*)$ such that $\overline {\psi}:X\rightarrow Y$ is a continuous map and $\psi^*:\mathcal O_Y\rightarrow \mathcal {\overline {\psi}}_*\mathcal O_X$ is a
homomorphism between the sheaves of supercommutative $\mathbb Z_2$-graded
rings. Let $U\subset R^{m}$, the super ring space $U^{m|n}:=\big(U,C^\infty_U\otimes \wedge \mathbb R^n\big)$ is called a superdomain where $C^\infty_U$ is the sheaf of smooth functions on $U$.

A supermanifold of dimension $m|n$ is a super ringed space $(\overline M,\cO_M)$ that is locally isomorphic to $\mathbb R^{m|n}$ and $\overline M$ is a second countable and Hausdorff topological space. A morphism between two supermanifolds $M=(\overline M,\cO_M)$ and $N=(\overline N,\cO_N)$ is just a morphism between two super ringed spaces.
For each supermanifold $M$, one can show that there exists a smooth manifold associated to $M$ which is denoted by $\widetilde{M}$ and is called reduced manifold.

\medskip

The category of supermanifolds, $\textbf{SM}$, is a locally small category and it has finite product property. Also, it has a terminal object $\mathbb{R}^{0|0}$, that is the constant sheaf $\mathbb{R}$ on a singleton $\{0\}$.

\medskip

Let $M=(\overline M,\cO_M)$ be a supermanifold and $p\in \overline M$. There exists the map $j_p=(\overline {j_p},{j_p}^*)$ as follows:
$$\begin{matrix} \overline {j_p}:\{0\} \rightarrow \overline M\qquad\qquad 
& {j_p}^*:\mathcal{O}_M\rightarrow \mathbb{R}, \\
& \qquad\qquad\qquad\quad g \mapsto \tilde g(p)=:ev_p(g).\end{matrix}$$
So, for each supermanifold $T$, one can define the morphism
\begin{align} \hat{p}_{_{T}}:T\rightarrow \mathbb{R}^{0|0}\xrightarrow{j_p}M\label{phat},\end{align}
as a composition of $j_p$ and the unique morphism $T\rightarrow \mathbb{R}^{0|0}$.

\medskip

\begin{definition}
A super Lie group $G$ is a supermanifold $G$ together the morphisms  $\mu:G\times G \rightarrow G,\quad i:G \rightarrow G,\quad e:\mathbb{R}^{0|0} \rightarrow G$ say multiplication, inverse and unit morphisms respectively, such that the following equations are satisfied
\begin{align*}
\mu \circ (\mu\times 1_{_{_G}})&=\mu \circ (1_{_{_G}}\times \mu)\\
\mu \circ (1_{_{_G}}\times \hat{e}_{_{_G}})\circ \bigtriangleup_{_{_G}} &=1_{_{_G}}=\mu \circ (\hat{e}_{_{_G}}\times 1_{_{_G}})\circ \bigtriangleup_{_{_G}}\\
\mu \circ (1_{_{_G}}\times i)\circ \bigtriangleup_{_{_G}}&=\hat{e}_{_{_G}}=\mu \circ (i\times 1_{_{_G}})\circ \bigtriangleup_{_{_G}}
\end{align*}
where
 $1_{_G}$ is identity on $G$ and $\hat{e}_{_G}$ is the morphism according to (\ref{phat}) for element $e\in \overline{G}$ and also $\bigtriangleup_{_{_G}}$ be the diagonal map on $G$.
\end{definition}

\medskip

Let $G$ be a super Lie group, then the reduced manifold $\widetilde{G}$ is a Lie group with multiplication, inverse and unit morphisms induced by $\mu$, $i$ and $e$ in above definition. Also, every super Lie group $G$ induces a group structure over its $T$-points for any arbitrary supermanifold $T$. Then, the functor $T\rightarrow G(T)$ takes values in category of groups. In addition, if $S$ is another supermanifold and $T\rightarrow S$ is a morphism, then the corresponding map $G(S)\rightarrow G(T)$ is a homomorphism of groups. One can also define a super Lie group as a representable functor $T\rightarrow G(T)$ from category $\textbf{SM}$ to category of groups. If such functor represented by a supermanifold $G$, then the maps $\mu,i,e$ are obtained by Yoneda's lemma and the maps $\mu_{_T}, i_{_T}$ and $e_{_T}$.\\


For example, let $V$ be a finite dimensional super vector space of dimension $m|n$ and let $\{R_1,\cdots,R_{m+n}\}$ be a  basis of V for which the first $m$ elements are even and the last $n$ elements are odd.
Consider the functor 
\begin{align*}
F: & \textbf{SM} \rightarrow \textbf{Grp}\\
& T\mapsto Aut_{\mathcal O(T)}(\mathcal O(T)\otimes V),\end{align*}
where $F$ maps each seupermanifold $T$ to the group of even $\mathcal O(T)$-module automorphisms of $\mathcal O(T)\otimes V$, and $\textbf{Grp}$ is the category of groups. Consider the supermanifold $\textbf{End}(V)=\Big(End(V_{0})\times End(V_{1}),\mathcal{A}\Big)$ where $\mathcal{A}$ is the sheaf $C^{\infty}_{\mathbb{R}^{m^2+ n^2}}\otimes \wedge{\mathbb{R}^{2mn}}$. Let 
$F_{ij}$ is a linear transformation on $V$ with $R_k\mapsto \delta_{ik}R_j$. If $\{f_{ij}\}$ is the corresponding dual basis, then it may be considered as a global coordinates on $\textbf{End}(V)$. Let $X$ be the open subsupermanifold of $\textbf{End}(V)$ corresponding with the open set:
$$\overline X=GL(V_{0})\times GL(V_{1})\subset End(V_{0})\times End(V_{1}).$$
Thus, we have
$$X=\Big(GL(V_{0})\times GL(V_{1}),\mathcal{A}|_{GL(V_{0})\times GL(V_{1})}\Big).$$
It can be shown that the functor $F$ may be represented by $X$. For this, one may show that $Hom(T,X)\cong Aut_{\mathcal O(T)}(\mathcal O(T)\otimes V)$. To this end, first, note that
$$Hom(T,X)=Hom(\mathcal A(X),\mathcal O(T)).$$
 It is known that each $\psi\in Hom(\mathcal A(X),\mathcal O(T))$ may be uniquely determined by $\{g_{ij}\}$ where $g_{ij}=\psi(f_{ij})$, see \cite{vsv}.
 Now set $\Psi(R_j):=\Sigma g_{ij}R_i$.
 One may consider $\Psi$ as an element of $Aut_{\mathcal{O}(T)}(\mathcal{O}(T)\otimes V)$. Obviously $\psi\mapsto\Psi$ is a bijection from $Hom(T,X)$ to $Aut_{\mathcal{O}(T)}(\mathcal{O}(T)\otimes V)$.
Thus the supermanifold $X$ is a super Lie group and denoted it by $GL(V)$ or $GL(m|n)$ if $V=\mathbb{R}^{m|n}$.
Therefore $T$- points of $GL(m|n)$ are the invertible even $m|n\times m|n$ supermatrices $\begin{pmatrix}
A & B\\ C & D \end{pmatrix}$ such that
$a_{ij}, d_{kl}\in \mathcal O(T)_0,\quad b_{il},c_{kj}\in \mathcal O(T)_1 $ and the multiplication can be written as the matrix product.

There exists, up to isomorphism, a unique super Lie algebra, say $\mathfrak g$, associated to super Lie group $G$. Then one may write $\mathfrak g=Lie(G)$. The super Lie algebra of $G$ may be thought of all right (or left) invariant vector fields on $G$.   

\begin{definition}
Let $M$ be a supermanifold and let $G$ be a super Lie group with $\mu,i$ and $e$ as its multiplication, inverse and unit morphisms respectively. A morphism $a:M\times G\rightarrow M$ is called a (right) action of $G$ on $M$, if
$$\begin{matrix} a\circ(1_M\times\mu)=a\circ(a\times 1_{_G}),&\qquad\qquad & a\circ(1_M\times\hat{e}_{_M})\circ \Delta_M=1_M. \end{matrix}$$
where $\hat{e}_{_M}, \Delta_M$ are as above. In this case, we say G acts from right on $M$. One can define left action analogously. 
\end{definition}
By Yoneda lemma, one may consider, equivalently, the action of G as a natural transformation: 
$$a(.):M(.)\times G(.)\rightarrow M(.),$$
such that for each supermanifold $T$, the morphism $a_{_T}: M(T)\times G(T)\rightarrow  M(T)$ is an action of group $G(T)$ on the set $M(T)$. This means for each $\QQ_1,\QQ_2\in G(T)$ and $\PP\in M(T)$ one has
\begin{enumerate}
\item[1.] $(\PP.\QQ_1).\QQ_2=\PP.(\QQ_1\QQ_2).$
\item[2.] $\PP.\hat{e}_{_T}=\PP.$
\end{enumerate}

Let  $p\in \overline M$, define  
\begin{align*} & a_p:G\rightarrow M, \qquad\qquad\qquad\qquad\qquad a^g:M\rightarrow M,\\ & a_p:=a\circ (\hat{p}_{_G}\times 1_{_G})\circ \Delta_{_G}, \qquad\qquad a^g:=a\circ (1_M \times \hat{g}_{_M})\circ \Delta_M, \end{align*}
where $\hat{p}_{_G}$ and $\hat{g}_{_M}$ are the morphism (\ref{phat}) for $p\in \overline M$ and $g\in \overline{G}$ respectively. Equivalently, these maps may be defined as
\begin{align*} (a_p)_{_T}:& G(T)\rightarrow M(T), \qquad\qquad\qquad (a^g)_{_T}:M(T)\rightarrow M(T),\\ & \QQ\longmapsto \hat{p}_{_T}.\QQ, \qquad\qquad\qquad\qquad\qquad\quad \PP\longmapsto \PP.\hat{g}_{_T}. \end{align*}
Before next definition, we recall
that a morphism between supermanifolds, say $\psi:M\rightarrow N$ is a submersion at $p\in\overline M$, if $(d\psi)_p$ is surjective and $\psi$ is called submersion, if this happens at each point. (For more detail about this, One can refer to \cite{carmelibook}, \cite{vsv}). $\psi$ is a \textit{surjective submersion}, if in addition $\overline{\psi}$ is surjective. Let $G$ acts on $M$ with action $a:M\times G \rightarrow M$. We say that $a$ is transitive, if there exist $p\in \overline M$ such that $a_p$ is a surjective submersion.  

It is shown that, if $a_p$ be a submersion for one $p\in \overline M$, then it is a submersion for all point in $\overline M$.
The following proposition will be required in the last section.
\begin{proposition}\label{Transitive}
Let $a:M\times G \rightarrow M$ be an action. Then $a$ is transitive if and only if $(a_p)_{\mathbb{R}^{0|q}}:G(\mathbb{R}^{0|q})\rightarrow M(\mathbb{R}^{0|q})$ is surjective, where q is the odd dimension of $G.$
\end{proposition}
\begin{proof} See the proof of proposition 9.1.4 in \cite{carmelibook}. \end{proof}
\begin{definition}
Let $G$ be a super Lie group and let $a$ be an action of $G$ on supermanifold $M$. By \textit{stabilizer} of $p\in \overline M$, we mean a supermanifold $G_p$ equalizing the diagram $$G\overset{a_p}{\underset{\hat{p}_{_G}} {\rightrightarrows}}M.$$
 \end{definition}
It is not clear that such an equalizer exists. In this regard, there are two propositions as follow
\begin{proposition}\label{Isotropic}
Let $a:M\times G \rightarrow M$ be an action, then
\begin{enumerate} 
\item[1.] The  following diagram admits an equalizer $G_p$ 
$$G\overset{a_p}{\underset{\hat{p}_{_G}} {\rightrightarrows}}M.$$
\item[2.] $G_p$ is a sub super Lie group of $G$.
\item[3.] The functor $T\rightarrow (G(T))_{\hat{p}_{_T}}$ is represented by $G_p$, where $(G(T))_{\hat{p}_{_T}}$ is the stabilizer of $\hat{p}_{_T}$ for the action of $G(T)$ on $M(T)$.
\end{enumerate} \end{proposition}
\begin{proof} See the proof of proposition 8.4.7 in \cite{carmelibook}. \end{proof}
 \begin{proposition}\label{equivariant}
 Suppose $G$ acts transitively on $M$. Then there exists a $G$-equivariant isomorphism 
 \begin{displaymath}\xymatrix{ \dfrac{G}{G_p}\ar[r]^{\cong} & M. } \end{displaymath} \end{proposition} 
 \begin{proof} See the proof of proposition 6.5 in \cite{g-spaces}. \end{proof}
 \subsection{Super Harish-Chandra pairs}
 In this section, we recall \textit{super Harish-Chandra pairs} and their morphisms as a category equivalent to the category of super Lie groups.
\begin{definition}
Let $G_0$ be a Lie group and let $\mathfrak g=\mathfrak g_0\oplus\mathfrak g_1$ be a super Lie algebra such that
 \begin{enumerate}
  \item[(a)]  $\mathfrak g_0\cong Lie(G_0)$
  \item[(b)]  $G_0$ acts on $\mathfrak g$ via a representation $\sigma$ such that $\sigma(G)_{|_{\mathfrak g_0}}=Ad$ and the differential
   of $\sigma$ acts on $\mathfrak g$ as the adjoint representation, that is,
  $$d\sigma(X)Y=[X,Y].$$
  \end{enumerate}
Then the triple $(G_0,\mathfrak g,\sigma)$ or pair $(G_0,\mathfrak g)$ is called a super Harish-Chandra pair.
 \end{definition}
 
 A morphism between two super Harish-Chandra pairs is a pair of morphisms $(\psi_0,\rho^{\psi})$ that preserves the structure of super Harish- Chandra pair. More precisely,

\begin{definition}
 Let  $(G_{0},\mathfrak{g},\sigma)$ and $(H_{0},\mathfrak{h},\tau)$ be two super Harish-Chandra pairs. A morphism between them
 is a pair $(\psi_{0},\rho^{\psi})$ such that
 \begin{enumerate}
   \item[(1)]  $\psi_{0}:G_{0}\rightarrow H_{0}$ is a Lie group homomorphism. 
   \item[(2)]  $\rho^{\psi}:\mathfrak{g} \rightarrow \mathfrak{h}$ is a super Lie algebra homomorphism.
   \item[(3)]   $\psi_{0}$ and $\rho^{\psi}$ are compatible, i.e.
        $$\rho^{\psi} \mid_{\mathfrak{g}_{0}}=d\psi_{0}.$$
   \item [4)]   For each $x\in G_{0},$ 
   $$\rho^{\psi}\circ \sigma(x)=\tau(\psi_{0}(x))\circ \rho^{\psi}.$$
   \end{enumerate}
\end{definition}

 For example, let $G$ be a super Lie group, the pair $(\widetilde{G},\mathfrak g)$ given by the reduced Lie group $\widetilde{G}$ of $G$ and the associated super Lie algebra $\mathfrak g$ is a super Harish-Chandra pair with respect to the
 adjoint action of $\widetilde{G}$ on $\mathfrak g$ of $G$.
One can show that the super Harish-Chandra pairs with their morphisms form a category that is denoted by $\textbf{Shcp}.$ Also, the functor
  \begin{align*}
     F:  \textbf{SGrp}  & \longrightarrow  \textbf{Shcp}
     \\
           G & \longmapsto (\widetilde{G},Lie(G),Ad)
          \\
           \psi & \longmapsto (\widetilde{\psi},d\psi)
      \end{align*}
 is bijective. See \cite{carmelibook} for more details.
 
The following definition defines an action of a super Harish-Chandra pair on a supermanifold. 
 \begin{definition}
 We say that a super Harish-Chandra pair $(G_0,\mathfrak g)$ acts on a supermanifold $M$ if there is a pair $(\underline{a},\rho)$ such that the following holds:
 \begin{enumerate}
 \item[(1)]
 $\underline{a}:M \times G_0  \rightarrow M$ is an action of $G_0$  on $M.$ 
 \item[(2)]
 $\rho:\mathfrak g\rightarrow \textbf{Vec}(M)$ is a super Lie algebra morphism such that
 \begin{enumerate}
 \item[(a)]
 $\rho|_{\mathfrak g_0}(X)\simeq (1_{\mathcal O(M)}\otimes X)\underline{a}^*,\qquad \forall X\in \mathfrak g_0,$
 \item[(b)]
 $\rho(g.Y)=(\underline{a}^{g^{-1}})^*\rho(Y)(\underline{a}^{g})^*,\qquad \forall g\in\overline{G},Y\in \mathfrak g,$
 \end{enumerate} \end{enumerate}
 where $\textbf{Vec}(M)$ is the super Lie algebra including all vector fields on $M.$
 \end{definition}
 \begin{proposition}\label{prop1}
 Let $a:M \times G\rightarrow M$ be an action of the super Lie group $G$ 
 on the supermanifold $M$ , then the morphisms
 \begin{enumerate}
 \item[(1)]
 $\underline{a}:M\times \widetilde{G} \rightarrow M,$ with
 $\underline{a}=a\circ (1_M \times j),$ where
 $j:\widetilde{G}\rightarrow G$
 is the natural
 immersion of the reduced Lie group $\widetilde G$ into $G$,
 \item[(2)]
 $\rho_a:\mathfrak g\rightarrow \textbf{Vec}(M)$, $Y\mapsto (1_{\mathcal O(M)}\otimes Y)a^*$
 \end{enumerate}
 define an action of the associated super Harish-Chandra pair $(\widetilde{G},\mathfrak g)$ on M .
 \end{proposition}
 \begin{proposition}\label{prop2}
 Let $(\widetilde{G},\mathfrak g)$ be the super Harish-Chandra pair associated with the super Lie group $G$ and
 let $(\underline{a},\rho)$ be an action of $(\widetilde{G},\mathfrak g)$ on a supermanifold $M$. Then there is a unique action $a_{\rho}:M\times G\rightarrow M$ of the super Lie group $G$ on $M$ whose reduced and infinitesimal actions are $(\underline{a},\rho).$
 \end{proposition}
 See \cite{carmelibook} for proofs of these two propositions.
 \begin{definition}\label{invariantvector}
 We say $Y\in \textbf{Vec}(M)$ is right invariant vector field associated to action $a$, if there exists a right invariant vector field $X$ on $G$ such that $\rho_a(X)=Y.$ 
 \end{definition}
 We denote the set of all such vector fields on $M$ by $\textbf{Vec}_R(M)$
\section{
$\nu$
-grassmannians}\label{nugrasssection}
Supergrassmannians are introduced by Manin in \cite{ma1}. In this section, we study $\nu$-grassmannians, a new generalization of Grassmannians in supergeometry. For this, we recall some concepts and definitions from \cite{OurArticle} and \cite{vclass}.
By a $\nu$-domain, we mean a superdomain $\mathbb R^{m|n}$ with an odd involution
$$\nu:C^\infty_{\mathbb R^m}\otimes \wedge\mathbb R^n\rightarrow C^\infty_{\mathbb R^m}\otimes \wedge\mathbb R^n$$
as a morphism of $C^\infty_{\mathbb R^m}$-modules, i.e. $\nu^2=1$ and
$$\nu\Big(C^\infty_{\mathbb R^m}\otimes \wedge^\circ\mathbb R^n\Big)\subseteq C^\infty_{\mathbb R^m}\otimes \wedge^e\mathbb R^n,\qquad \nu\Big(C^\infty_{\mathbb R^m}\otimes \wedge^e\mathbb R^n\Big)\subseteq C^\infty_{\mathbb R^m}\otimes \wedge^\circ\mathbb R^n$$
where $\wedge^e\mathbb R^n$ and $\wedge^\circ\mathbb R^n$ are even and odd parts of $\mathbb Z_2$-graded ring $\wedge\mathbb R^n$ respectively, and
$$\nu(f\xi)=f\nu(\xi),\qquad \forall f\in C^\infty(\mathbb R^m),\xi\in \wedge\mathbb R^n.$$
This shows that the structure sheaf carries a $\mathbb R[\nu_\circ]$-module structure where $R[\nu_\circ]$ is the
polynomials ring generated by $\nu_\circ$ with $\nu_\circ^2=1.$ In fact, the action of $\mathbb R[\nu_\circ]$ is completely determined by
$\nu_\circ\xi:=\nu(\xi)$ for each $\xi\in\wedge\mathbb{R}^n$.
We denote a $\nu$-domain of dimension $m|n$ by $_\nu R^{m|n}$. Although not being unique, but such an involution exists. See \cite{OurArticle}, \cite{vclass} and \cite{vmanifold} for more details.
A morphism between two $\nu_\circ$-domains is a morphism between two super ringed
spaces, such that it preserves the above $R[\nu_\circ]$-module structure. 

The notion of $\nu$-Grassmannians is introduced in \cite{OurArticle}. which we shall now discuss briefly. By a $\nu$-Grassmannian $_\nu G_{k|l}(m|n) $, we mean a supermanifold, with structure sheaf $\mathcal O$, which is constructed by gluing $\nu$-domains
$_{\nu}\mathbb R^{\alpha|\beta}=\big(\mathbb{R}^{\alpha}, C^{\infty}_{\mathbb{R}^{\alpha}} \otimes \wedge \mathbb{R}^{\beta} \big) $ 
where $ \alpha=k(m- k)+ l(n-l) $ and $ \beta=l(m- k)+k(n- l) $.\\
Let $I\subset \{1, \cdots, m \}$ and $ R \subset \{1, \cdots, n\}$ be sorted subsets in ascending order, with $ p $ and $ q $ elements respectively such that $p+q=k+l$. The elements of $ I $ are called even indices and the elements of $ R $ are called odd indices. In this case, $I|R$ is called a $p|q$-index. Set
$\overline{U}_{I|R}=\mathbb{R}^{\alpha}, \mathcal{O}_{I|R}=\,  C^{\infty}_{\mathbb{R}^{\alpha}}\otimes \wedge \mathbb{R}^{\beta}.$
If $ p=k $ then $ I|R $ is called a \textbf{standard index} and $ (U_{I|R},O_{I|R}) $, or $ U_{I|R} $ for brevity, is called a \textbf{Standard domain} and otherwise they are called \textbf{non standard index} and \textbf{non standard domain} respectively. Decompose any even supermatrix into four blocks, say $ B_1, B_{2}, B_3, B_4 $. Upper left and lower right blocks, $ B_1, B_4 $ are $ k\times m $ and $ l\times n $ matrices respectively. They are called even blocks. Upper right and lower left blocks, $ B_2, B_3 $ are $ k\times n $ and $ l\times m $ matrices. They are called odd blocks. In addition, by even part, we mean the blocks $ B_1, B_3 $ and by odd parts we mean the blocks $ B_2, B_4 $. Blocks, $ B_1, B_4 $ are filled with even elements and blocks, $ B_2, B_3 $ are filled with odd elements. By \textbf{divider line}, we mean the line which separates odd and even parts.

Let each domain $ U_{I|R} $ be labeled by an even $ k|l\times m|n $ supermatrix, say $A_{I|R}$. Let $J\subset \{1, \cdots, m \}$ and $ S \subset \{1, \cdots, n\}$ be arbitrary subsets,
then by $M_{J|S}A_{I|R}$, we mean the matrix obtained by the columns of $A_{I|R}$ with indices in $J \cup S$ together. For each $A_{I|R}$, we also assume that, except for columns with indices in $ I \cup R $, which together form a minor denoted by $ M_{I|R}A_{I|R} $, the even and odd blocks are filled from up to down and left to right by $ x_a^I, e_b^I $, the even and odd coordinates of $ U_{I|R} $ respectively, i.e. $ (x^I_a)_{1\leq a\leq \alpha} $ is a coordinates system on $ \mathbb{R}^\alpha $ and $ \{e^I_b\}_{1\leq b\leq \beta} $ is a basis for $\mathbb{R}^{\beta} $.
 This process imposes an ordering on the set of coordinates as follows:
\begin{align}\label{ordering}
x_1,\ldots,x_k,e_1,\ldots,e_l,\ldots,x_{(m-k-1)k+1},\ldots,x_{(m-k)k},e_{(m-k-1)l+1},\ldots,e_{(m-k)l},\nonumber\\
e_{(m-k)l+1},\ldots, e_{(m-k)l+k}, x_{(m-k)k+1},\ldots,x_{(m-k)k+l},\ldots,
e_{(m-k)l+(n-l-1)k+1}\nonumber\\ ,\ldots, e_{(m-k)l+(n-l)k}, x_{(m-k)k+(n-l-1)k+1},\ldots,x_{(m-k)k+(n-l)l}. 
\end{align} 
If $ p=k $ then $ M_{I|R}A_{I|R} $ is supposed to be an identity matrix. Let $ I|R $ and $ J|S $ be two standard indices and let $  M_{J|S}A_{I|R} $ be the minor consisting of columns of $ A_{I|R} $ with indices in $ J\cup S $. By $ U_{{I|R},{J|S}} $ we mean the set of all points of $ \overline{U}_{I|R} $, on which $ M_{J|S}A_{I|R} $ is invertible. Obviously $ U_{{I|R},{J|S}} $ is an open set.  

For example, let $ I=\{1\}, R=\{2,3\} $ and let $ I|R $ be a $ 1|2 $-index in $ _\nu G_{1|2}(2|3) $. In this case, the set of coordinates of $\mathcal{O}_{I|R} $ is 
\[
\lbrace x_1, x_2, x_3; e_{1}, e_{2}, e_{3}\rbrace,
\]

and $ A_{I|R} $ is:
\begin{equation*}
\left[
\begin{array}{cc|ccc}
1 & x_1 & e_3 & 0 & 0 \\
\hline
0 & e_1 & x_2 & 1 & 0 \\
0 & e_2 & x_3 & 0 & 1 \\
\end{array}
\right].
\end{equation*}
Thus $ \{x_1, e_1, e_2, e_3, x_2, x_3\} $ is the corresponding total ordered set of generators.

If $ p\neq k $, then $ M_{I|R}A_{I|R} $ is a $ k|l\times p|q $  supermatrix as follows:\\
Let $ M_{I|R}A_{I|R} $ be partitioned into four blocks $ B_i, i=1,2, 3, 4 $, as above. All entries of this supermatrix except diagonal entries are zero. In addition, diagonal entries are equal to 1 if they place in $ B_1 $ and $ B_4 $ and are equal to $ 1\nu $ if they place in $ B_2 $ and $ B_3 $, where $ 1\nu $ is a formal  symbol. One may consider it as 1 among odd elements. Nevertheless we learn how to deal with it as we go further. Such supermatrix is called a \textbf{non-standard identity}. All places in $ A_{I|R} $ except for $ M_{I|R}A_{I|R} $ are filled by coordinates $ x^I $ and $ e^I $ according to the ordering, as stated above, from up to down and left to right.  In this process, if an even element, say $ x $, places in odd part then it is replaced by  $ \nu(x) $ and if an odd element, say $ e $, places in even part then it is replaced by  $ \nu(e) $.

As an example, consider $ _\nu G_{1|2}(2|3) $ and let $ I=\{1, 2\}, R=\{2\}$, so $ I|R $ is an $ 2|1 $-index. In this case, $ A_{I|R} $ is as follows:
\begin{equation*}
\left[
\begin{array}{cc|ccc}
1 & 0 & \nu(x_1) & 0 & e_3 \\
\hline
0 & 1\nu & \nu(e_1) & 0 & x_2 \\
0 & 0 & \nu(e_2) & 1 & x_3 \\
\end{array}\right].
\end{equation*}
Also if $ J=\emptyset, S=\{1,2,3\}$, then $ J|S $ is a $0|3$-index and $A_{J|S}$ is as follows:
\begin{equation*}
\left[
\begin{array}{cc|ccc}
x_1 & \nu(e_3) & 1\nu & 0 & 0 \\
\hline
e_1 & \nu(x_2) & 0 & 1 & 0 \\
e_2 & \nu(x_3) & 0 & 0 & 1 \\
\end{array}
\right].
\end{equation*}
Now, let us start by constructing morphisms $g_{I|R,J|S}$ through which the $ \nu $- domains $ U_{I|R} $ may be glued together. Set
\begin{equation}\label{equg}
g_{I|R,J|S} =(\overline{g}_{I|R,J|S}, g^*_{I|R,J|S}): (U_{I|R,J|S}, \mathcal{O}_{I|R}|_{U_{I|R,J|S}}) \rightarrow (U_{J|S,I|R},\mathcal{O}_{J|S}|_{U_{J|S,I|R}}). 
\end{equation}
First of all we introduce $ g^*_{{I|R}.{J|S}} $. For this, we should consider  two cases:
\begin{enumerate}
	\item[Case 1:]
If both domains are standard domains, then the transition map 
\begin{equation*}
g^*_{I|R,J|S}:  \mathcal{O}_{J|S}|_{U_{J|S,I|R}}\rightarrow  \mathcal{O}_{I|R}|_{U_{I|R,J|S}},
\end{equation*}
is obtained from the pasting equation:
\begin{equation*}
D_{J|S}\bigg(\big(M_{J|S}A_{I|R}\big)^{-1}A_{I|R}\bigg)=D_{J|S}A_{J|S},
\end{equation*}
where $ D_{J|S}A_{J|S} $ is a matrix which remains after omitting $ M_{J|S}A_{J|S} $. 
This equation defines $ g^*_{{I|R} ,{J|S}} $, for each entry of $ D_{J|S}(A_{J|S}) $, to be a rational expression in generators of $ \mathcal{O}_{I|R} $. This determines $ g^*_{{I|R},{J|S}} $ as a unique morphism (\cite{vsv}, Theo. 4.3.1).
Clearly, this map is defined whenever $ M_{J|S}A_{I|R} $ is invertible.
	\item[Case 2:]
Let $ U_{I|R} $ be an arbitrary domain, and $ U_{J|S} $ be a non-standard domain, and let $ A_{I|R} $ and $ A_{J|S} $ be their labels respectively. Moreover, let $ J|S $ be a $ p|q $-index such that $ p\neq k $. In this case, $M_{J|S}A_{I|R}$ is a $ k|l\times p|q $ supermatrix. Let $ M'_{J|S}A_{I|R} $ be a $ k|l\times k|l $ supermatrix associated to $ M_{J|S}A_{I|R} $ as follows:\\
Consider the columns of non-standard identity $ M_{J|S}A_{J|S} $ which contains $ 1\nu $.  Move the columns in $ M_{J|S}A_{I|R} $ with the same indices as the columns in $  M_{J|S}A_{J|S} $ which contain $ 1\nu $ to another side of the divider line of $  M_{J|S}A_{I|R} $ and replace each entry, say $ a $, in these columns with $ \nu(a) $. The resulting matrix is denoted by $ M'_{J|S}A_{I|R} $.
\end{enumerate}
For example in $_\nu G_{1|2}(2|3)$ suppose $I=\{1\}, R=\{1,2\}, J=\{1,2\}, S=\{3\}$ , so $I|R$ is a $1|2$-index and $J|S$ is a $2|1$-index. We have
\begin{equation*}
A_{I|R}=\left[
\begin{array}{cc|ccc}
1 & x_1 & 0 & 0 & e_3\\
\hline
0 & e_1 & 1 & 0 & x_2\\
0 & e_2 & 0 & 1& x_3\\
\end{array}
\right], \quad 
A_{J|S}=\left[
\begin{array}{cc|ccc}
1 & 0 & \nu(x_1) & e_3 & 0 \\
\hline
0 & 1\nu & \nu(e_1) & x_2 & 0 \\
0 & 0 & \nu(e_2) & x_3 & 1 \\
\end{array}
\right],
\end{equation*}
\begin{equation*}
M_{J|S}A_{I|R}=\left[
\begin{array}{cc|c}
1 & x_1 & e_3 \\
\hline
0 & e_1 & x_2 \\
0 & e_2 & x_3 \\
\end{array}
\right], \qquad 
M^\prime_{J|S}A_{I|R}=\left[
\begin{array}{c|cc}
1 & \nu x_1 & e_3 \\
\hline
0 & \nu e_1 & x_2 \\
0 & \nu e_2 & x_3 \\
\end{array}
\right].
\end{equation*}
Let $ I|R $ be a standard index and $ J|S $ be a non standard index and let $M^\prime_{J|S}A_{I|R}$ be as above. By $ U_{{I|R},{J|S}} $ we mean the set of all points of $ U_{I|R} $, on which
$ M^\prime_{J|S}A_{I|R} $ is invertible. Obviously, $ U_{{I|R},{J|S}} $ is an open set.  

Now, we can define a coordinate transformation:
\begin{equation*}
g^*_{I|R,J|S}:  \mathcal{O}_{J|S}|_{U_{J|S,I|R}}\rightarrow  \mathcal{O}_{I|R}|_{U_{I|R,J|S}}.
\end{equation*}
This map is obtained from the following equation:
\begin{equation*}
D_{J|S}\bigg(\big(M^\prime_{J|S}A_{I|R}\big)^{-1}A_{I|R}\bigg)=D_{J|S}A_{J|S}.
\end{equation*}
It can be shown that the sheaves on $ U_{I|R}$ and $U_{J|S} $ can be glued through these maps. By (\cite{vsv}, page 135), we have to show the next proposition.
\begin{proposition}\label{prop1}
	Let $ g^*_{I|R, J|S} $ be as above, then
	\begin{align*}
	&1.\, g^*_{I|R,I|R}=id.\\
	&2.\, g^*_{I|R,J|S}\circ g^*_{J|S,I|R}=id.\\
	&3.\, g^*_{I|R,J|S}\circ g^*_{J|S,T|P}\circ g^*_{T|P,I|R}=id.
	\end{align*}
\end{proposition}
\begin{proof}
See \cite{OurArticle}.
\end{proof}
Let $(X , \mathcal O)$ be the ringed space which is constructed by gluing $(U_{I|R}, \mathcal O_{I|R})$ through $g_{I|R,J|S},$ 
then its reduced manifold is diffeomorphic to 
$G_k(\mathbb R^m) \times G_l(\mathbb R^n)$. See Proposition 2.3 in \cite{OurArticle}.
\section {
$\nu$
-grassmannian as homogeneous superspace}
In this section, we want to show that the $\nu$-grassmannian $_\nu G_{k|l}(m|n)$ is a homogeneous superspace.
 According to \cite{g-spaces}, \cite{qoutiontsuper} and \cite{carmelibook}, it is sufficient to show that the super Lie group $GL(m|n)$ acts on $_\nu G_{k|l}(m|n)$ transitively.

In \cite{MyArticle}, we have shown that $GL(m|n)$ acts on supergrassmannian $G_{k|l}(m|n)$ transitively. In this article, we use the same way as mentioned in \cite{MyArticle} for constructing the action of $GL(m|n)$ on $_\nu G_{k|l}(m|n).$ In the following, we construct the right action of $GL(m|n)$ on $_\nu G_{k|l}(m|n)$.

By Yoneda lemma, for defining a morphism
$a:\,_{\nu}G_{k|l}(m|n)\times GL(m|n) \rightarrow \,_{\nu}G_{k|l}(m|n),$
it is sufficient, to define $$a_{_T}:\,_{\nu}G_{k|l}(m|n)(T)\times GL(m|n)(T) \rightarrow \,_{\nu}G_{k|l}(m|n)(T),$$
for each supermanifold $T$, or equivalently define 
$$(a_{_T})^\PP:\, _{\nu}G_{k|l}(m|n)(T) \rightarrow \,_{\nu}G_{k|l}(m|n)(T),$$
where $\PP$ is a fixed arbitrary element in $GL(m|n)(T)$. From now, we denote $(a_{_T})^\PP$ by $\textbf{A}$.
Since $GL(m|n)$ is a superdomain, each element $\PP\in GL(m|n)(T)$ may be considered as an invertible $m|n\times m|n$ supermatrix with entries in $\mathcal{O}(T),$ as stated in the subsection \ref{subsection2} and it is denoted by $[\PP]$.
But it is not the case for $_{\nu}G_{k|l}(m|n)(T).$ So the action of $GL(m|n)$ on $_{\nu}G_{k|l}(m|n)$ should be constructed  locally.
Indeed, one may define the actions of $GL(m|n)$ on $\nu$-domains $(\overline U_{I|R},\mathcal{O}_{I|R})$ and then show that these actions glued together to construct the action $\textbf{A}$.  

Let $\XX$ be an element of $U_{I|R}(T)$ where $I|R$ is an arbitrary index.
We have to introduce the $k|l\times m|n$ supermatrix, denoted by $[\XX]_{I|R}$ corresponding to the morphism $\XX$ as follows:
Except for columns with indices in $ I \cup R $, the even and odd blocks are filled from up to down and left to right by $ f_i, g_j $'s
$$f_i:=\XX(x_i), \quad g_j:=\XX(e_j),$$
according to the ordering (\ref{ordering}), where $(x_i; e_j)$ is the global coordinates of the superdomain $U_{I|R}$. The columns with indices in $ I \cup R $ form an identity and non-standard identity matrix, if $p=k$ and $p\neq k$ respectively. 

For defining the morphism $\textbf{A}$, it is needed to refine the covering $\{U_{I|R}(T)\}_{I|R}$. Set
$$U_{I|R}^{J|S}(T):=\Big\{X\in U_{I|R}(T)\quad|\quad M^\prime_{J|S}([X]_{I|R}[\PP])\qquad\text{is invertible}\Big\}$$
where $[\PP]$ is the supermatrix corresponding to $\PP$, an arbitrary element of GL(T), $[\XX]_{I|R}$ is the supermatrix as stated above and $M^\prime_{J|S}A$ is the $k|l\times k|l$ submatrix of $A$ formed by columns with indices in $I\cup R$. Obviously, $\{U_{I|R}^{J|S}(T)\}_{I|R,J|S}$ is a covering for $_{\nu}G_{k|l}(m|n)(T).$ 
Now consider the maps 
\begin{align*} 
 \textbf{A}_{I|R}^{J|S}: & U_{I|R}^{J|S}(T)\rightarrow U_{J|S}(T)\\
 & \quad X \rightarrow  D_{J|S}\Big(\big(M^\prime_{J|S}([X]_{I|R}[\PP])\big)^{-1}[X]_{I|R}[\PP]\Big).
 \end{align*}
The following commuted diagram shows that these maps can glue together to construct a global map on $_{\nu}G_{k|l}(m|n)(T)$.
\begin{Small} \begin{displaymath} \xymatrix{ & U_{I|R}^{J|S}(T)\cap U_{K|Q}^{H|L}(T) \ar[rd]^{(g_{{K|Q},{I|R}})_{_T}}\ar[ld]_{\textbf{A}_{I|R}^{J|S}} & \\ U_{J|S}(T)\cap U_{H|L}(T)\ar[dr]_{(g_{{H|L},{J|S}})_{_T}} &  & U_{I|R}^{J|S}(T)\cap U_{K|Q}^{H|L}(T) \ar[ld]^{\textbf{A}_{K|Q}^{H|L}}\\
& U_{J|S}(T)\cap U_{H|L}(T) & }\quad 
\end{displaymath} \end{Small}
where $\big(g_{I|R, J|S}\big)_{_T}$
is the induced map from $g_{I|R, J|S}$ on $T$-points. The following lemma is used to show commutativity of the above diagram. 
\begin{lemma}
Let $\psi:T\rightarrow \mathbb{R}^{\alpha|\beta}$ be a $T$-point of $\mathbb{R}^{\alpha|\beta}$ and $(z_{tu})$ be a global coordinates of  $\mathbb{R}^{\alpha|\beta}$ with ordering as the one introduced in (\ref{ordering}). If $B=(\psi^*(z_{tu}))$ is the supermatrix corresponding to $\psi$, then the supermatrix corresponding to $\big(g_{I|R,J|S}\big)_{_T}(\psi)$ is as follows:
\begin{equation*}
 D_{I|R}((M_{I|R}[B]_{J|S})^{-1}[B]_{J|S}),
\end{equation*}
where $[B]_{J|S}$ is as above.
\end{lemma}
\begin{proof}
See Lemma 3.1 in \cite{MyArticle}.
\end{proof}
\begin{proposition}\label{diagramprop}
The above diagram commutes.
\end{proposition}
\begin{proof}
We have to show that
\begin{equation}\label{glueaction}
(g_{{H|L},{J|S}})_{_T}\circ \textbf{A}_{I|R}^{J|S}=\textbf{A}_{K|Q}^
{H|L} \circ (g_{{K|Q},{I|R}})_{_T},
\end{equation}
for arbitrary $p|q$-indices $I|R, J|S, K|Q, H|L.$
Let $\psi\in U_{I|R}^{J|S}(T)\cap U_{K|Q}^{H|L}(T)$ be an arbitrary element. One has $\psi\in U_{I|R}^{J|S}(T)$, so 
\begin{align*} D_{J|S}\Big(\big(M^\prime_{J|S}([\psi]_{I|R}[\PP])\big)^{-1}[\psi]_{I|R}[\PP]\Big)&\in U_{J|S}(T),\\
(g_{{H|L},{J|S}})_{_T}\Bigg(D_{J|S}\Big(\big(M^\prime_{J|S}([\psi]_{I|R}[\PP])\big)^{-1}[\psi]_{I|R}[\PP]\Big)\Bigg)&\in U_{H|L}(T).
\end{align*}
From left side of (\ref{glueaction}), we have:
\begin{align*}
(& g_{{H|L},{J|S}})_{_T}\circ \textbf{A}_{I|R}^{J|S}(\psi)\\
&=(g_{{H|L},{J|S}})_{_T}\Bigg(D_{J|S}\bigg(\big(M^\prime_{J|S}([\psi]_{I|R}[\PP])\big)^{-1}[\psi]_{I|R}[\PP]\bigg)\Bigg)\\
& =D_{H|L}\Bigg(\bigg(M^\prime_{H|L}\Big[D_{J|S}\big(\big(M^\prime_{J|S}([\psi]_{I|R}[\PP])\big)^{-1}[\psi]_{I|R}[\PP]\big)\Big]_{J|S}\bigg)^{-1}\\
&\qquad\qquad\qquad\qquad\qquad\qquad\qquad\Big[D_{J|S}\big(\big(M^\prime_{J|S}([\psi]_{I|R}[\PP])\big)^{-1}[\psi]_{I|R}[\PP]\big)\Big]_{J|S}\Bigg),
\end{align*}
where $[\psi]_{I|R}$ is as above. By substitute $M^\prime_{J|S}([\psi]_{I|R}[\PP])$ by $Z$ in the last expression, we obtain:
 \begin{align}
 & =D_{H|L}\Bigg(\bigg(M^\prime_{H|L}\Big[D_{J|S}\big(Z^{-1}[\psi]_{I|R}
 [\PP]\big)\Big]_{J|S}\bigg)^{-1}\Big[D_{J|S}\big(Z^{-1}[\psi]_{I|R}[\PP]
 \big)\Big]_{J|S}\Bigg)\nonumber\\
 & =D_{H|L}\Bigg(\bigg(M^\prime_{H|L}\Big[Z^{-1}D_{J|S}\big([\psi]_{I|R}[\PP]\big)\Big]_{J|S}\bigg)^{-1}\Big[Z^{-1}D_{J|S}\big([\psi]_{I|R}[\PP]\big)\Big]_{J|S}\Bigg)\label{tasavi}.
 \end{align}
If $J|S$ is a $p|q$-index with $p>k,$ then one has
$$[\psi]_{I|R}[\PP]=[C_1,\cdots, C_{j_1}, \cdots, C_{j_k},\cdots, C_{j_p}, \cdots, C_m|\bar{C}_1,\cdots, \bar{C}_{s_1},\cdots, \bar{C}_{s_q},\cdots, \bar{C}_n]$$
where $C_t$, $\bar{C}_t$ are the columns of $[\psi]_{I|R}[\PP].$ We can also write
 $$\begin{matrix} M_{J|S}\big([\psi]_{I|R}[\PP]\big)=[C_{j_1},C_{j_2}, \cdots, C_{j_k},C_{j_{k+1}},\cdots, C_{j_p}|\bar{C}_{s_1},\bar{C}_{s_2},\cdots, \bar{C}_{s_q}]\\
 
 \\
   Z:= M^{\prime}_{J|S}\big([\psi]_{I|R}[\PP]\big)=[C_{j_1},C_{j_2}, \cdots, C_{j_k}|\nu C_{j_{k+1}},\cdots, \nu C_{j_p},\bar{C}_{s_1},
   \bar{C}_{s_2},\cdots, \bar{C}_{s_q}]\end{matrix}$$
   So, the supermatrix $\Big[Z^{-1}D_{J|S}\big([\psi]_{I|R}[\PP]\big)\Big]_{J|S}$ is given as follows;
\begin{align*}
\big[Z^{-1}C_1,\cdots, Z^{-1}C_{j_1-1}, Id_{j_1}, Z^{-1}C_{j_1+1}
,\cdots,Z^{-1}C_{j_k-1}, Id_{j_k}, Z^{-1}C_{j_k+1},\cdots, Z^{-1}C_{j_p-1}\\, Id_{j_p}, Z^{-1}C_{j_p+1},\cdots, Z^{-1}C_m \big\vert 
Z^{-1}\bar{C_1},\cdots, Z^{-1}\bar{C}_{{s_1}-1},\bar{Id}_{s_1},
Z^{-1}\bar{C}_{{s_1}+1},\cdots, Z^{-1}\bar{C}_{{s_q}-1}\\, \bar{Id}_{s_q},Z^{-1}\bar{C}_{{s_q}+1},\cdots,Z^{-1}\bar{C_n}\big]
\end{align*}
where $Id_{j_r}$ is a columnar supermatrix with the only one nonzero entry equals $ 1 $ on $ r $-th row, if $ 1\leq r\leq k $,  and equals $ 1\nu $ if $ k+1\leq r\leq p $. In addition $ Id_{s_t} $ is a columnar supermatrix with the only one nonzero entry equals to $ 1 $ on the $ (k+t)- th $ row for each $ 1\leq t\leq q $.
One can write
\begin{align*}
\Big[Z^{-1}D_{J|S}\big(\tilde{\psi}_{I|R}[\PP]\big)\Big]_{J|S}= Z^{-1}[C_1, \cdots, C_{j_1-1}, Z(Id_{j_1}), C_{j_1+1}, \cdots, C_{j_k-1},Z(Id_{j_k})\\,C_{j_k+1}, \cdots, C_{j_p-1},Z(Id_{j_p}),C_{j_p+1}, \cdots, C_m \big\vert\bar{C}_1, \cdots, \bar{C}_{s_1-1}, Z(\bar{Id}_{s_1}), \bar{C}_{s_1+1}, \cdots\\, \bar{C}_{s_q-1}, Z(\bar{id}_{s_q}),\bar{C}_{s_q+1}, \cdots, \bar{C}_n]
\end{align*}
Due to  the rule $z.1\nu=\nu(z),$
where $z$ is an arbitrary element of $\mathcal O$ in \cite{OurArticle}, the following equalities hold:
\begin{align*}
Z(Id_{j_r})&=C_{j_r} \qquad\qquad\qquad \forall 1\leq r\leq k,\\
Z(Id_{j_r})&=(\nu C_{j_r})1\nu=C_{j_r}\qquad \forall k+1\leq r\leq p,\\
Z(Id_{s_t})&=\bar{C}_{s_t}\qquad\qquad\qquad \forall 1\leq t\leq q.
\end{align*}
Simply, it follows
 $$\Big[Z^{-1}D_{J|S}\big([\psi]_{I|R}[\PP]\big)\Big]_{J|S}=Z^{-1}\big([\psi]_{I|R}[\PP]
 \big).$$
Thus by substituting in (\ref{tasavi}), one gets
\begin{align*}&
=D_{H|L}\Bigg(\bigg(M^\prime_{H|L}\Big(Z^{-1}\big([\psi]_{I|R}[\PP]\big)\Big)\bigg)^{-1} Z^{-1}\big([\psi]_{I|R}[\PP]\big)\Bigg)\\
&
=D_{H|L}\Bigg(\bigg(Z^{-1}M^\prime_{H|L}\Big(\big([\psi]_{I|R}[\PP]\big)\Big)\bigg)^{-1} Z^{-1}\big([\psi]_{I|R}[\PP]\big)\Bigg)\\
&
=D_{H|L}\Bigg(\bigg(M^\prime_{H|L}\big([\psi]_{I|R}P\big)\bigg)^{-1}Z Z^{-1}\big([\psi]_{I|R}[\PP]\big)\Bigg)\\
&
=D_{H|L}\Bigg(\bigg(M^\prime_{H|L}\big([\psi]_{I|R}[\PP]\big)\bigg)^{-1}\big([\psi]_{I|R}[\PP]\big)\Bigg).
\end{align*}
For right side of (\ref{glueaction}), we have:
\begin{align*}
& \textbf{A}_{K|Q}^{H|L}\circ(g_{{K|Q},{I|R}})_{_T}(\psi)=\textbf{A}_{K|Q}^{H|L}\Bigg(D_{K|Q}\bigg((M^\prime_{K|Q}[\psi]_{I|R})^{-1}[\psi]_{I|R}\bigg)\Bigg)\\
&  =D_{H|L}\Bigg\{\Bigg(M^\prime_{H|L}\Big(\Big[D_{K|Q}\big((M^\prime_{K|Q}[\psi]_{I|R})^{-1}[\psi]_{I|R}\big)\Big]_{K|Q}[\PP]\Big)\Bigg)^{-1}\\
&\qquad\qquad\qquad\qquad\qquad\qquad\Big[D_{K|Q}\big((M^\prime_{K|Q}[\psi]_{I|R})^{-1}[\psi]_{I|R}\big)\Big]_{K|Q}[\PP]\Bigg\}.
\end{align*}
Similar computations as above show that the right side of the last equality may be written as follows:
%
\begin{align*}
D_{H|L}\Bigg(\Big(M^\prime_{H|L}([\psi]_{I|R}[\PP])\Big)^{-1}[\psi]_{I|R}[\PP]\Bigg).
\end{align*}
This complete the proof.
\end{proof}
Proposition \ref{diagramprop} shows that the super Lie group $GL(m|n)$ acts on $\nu$-grassmannian ${}_\nu G_{k|l}(m|n)$ by the morphism $a.$
we have to show, this action is transitive.
\begin{theorem}
$GL(m|n)$ acts on ${}_\nu G_{k|l}(m|n)$ transitively.
\end{theorem}
\begin{proof}
By proposition \ref{Transitive}, it is sufficient to show that 
$$(a_p)_{\mathbb{R}^{0|r}}:GL(m|n)(\mathbb{R}^{0|r}) \rightarrow {}_\nu G_{k|l}(m|n)(\mathbb{R}^{0|r})$$
is surjective, where $r=2mn$ is the odd dimension of $GL(m|n)$ and
$$p=(p_1,p_2)\in\overline U_{I|R}\subset G_k(m)\times G_l(n).$$ 
Let $\bar{p}_1$ and $\bar{p}_2$ be the matrices corresponding to subspaces $p_1$ and $p_2$ respectively.
As an element of ${}_\nu G_{k|l}(m|n)(T),$ one can represent $\hat{p}_{_T}$ as follows
$$\hat{p}_{_T}=\begin{pmatrix}\bar{p}_1 & 0\\0 & \bar{p}_2\end{pmatrix},$$
where $T$ is an arbitrary supermanifold. For surjectivity, let
$$W=\begin{pmatrix}
A & B\\ C & D \end{pmatrix}\in U_{J|S}(\mathbb{R}^{0|r}),$$
 be an arbitrary element. We have to show that, there is an element
$V\in GL(m|n)(\mathbb{R}^{0|r})$ which $\hat{p}_{_T}V=W.$ 
Because $\bar{p}_1$ is in $G_k(m)$ and also $A$ is a matrix with rank $k$ and the Lie group $GL(m)$ acts on the manifold $G_k(m)$ transitively then there exists an invertible matrix $H_{m\times m}\in GL(m)$ such that $\hat{p}_{_T}H=A.$
Similarly one may see that there exists an invertible matrix $Q_{n\times n}\in GL(n)$ that $\hat{p}_2Q=D$.
In addition, equations $\bar{p}_1Z=B$ and $\bar{p}_2Z^\prime=C$ have solutions. Let $M_{m\times n}$ and $N_{n\times m}$ be the solutions of these equation respectively. 
Clearly, one can see 
 $$V=\left[\begin{array}{c|c}H_{m\times m} & M_{m\times n}\\
 \hline N_{n\times m} & Q_{n\times n} \end{array}
 \right]_{m|n\times m|n}$$
satisfies in $\hat{p}_{_T}V=W.$ So $(a_p)_{\mathbb{R}^{0|r}}$ is surjective.
Thus by proposition \ref{Transitive}, $GL(m|n)$ acts on 
 ${}_\nu G_{k|l}(m|n)$ transitively.
\end{proof}
So ${}_\nu G_{k|l}(m|n)$ is a homogeneous superspace by proposition \ref{equivariant}. 
\section{A New Super Lie group}\label{newgroup}
Let ${}_\nu G_{k|l}(m|n)$ be a $\nu$-grassmannian and 
$ a:{}_\nu G_{k|l}(m|n)\times GL(m|n) \rightarrow {}_\nu G_{k|l}(m|n)$ 
be the action from the Lie super group $GL(m|n)$ on the supermanifold ${}_\nu G_{k|l}(m|n)$ introduced in the last section. According to proposition \ref{prop1}, there exist reduced and infinitesimal actions
 \begin{enumerate}
 \item[(1)]
 $\underline{a}:{}_\nu G_{k|l}(m|n) \times \big(GL(m)\times GL(n)\big) \rightarrow {}_\nu G_{k|l}(m|n).$
 \item[(2)]
 $\rho_a:\mathfrak{gl}(m|n)\rightarrow \textbf{Vec}({}_\nu G_{k|l}(m|n)).$
 \end{enumerate}
 Let $X\in\textbf{Vec}({}_\nu G_{k|l}(m|n))$ be a right invariant vector field associated to action $a$ (Definition \ref{invariantvector}), 
 and let $\underline{\textbf{Vec}}\big({}_\nu G_{k|l}(m|n)\big)$ be the set of all right invariant vector fields $X$ on ${}_\nu G_{k|l}(m|n)$ such that for each $\hat{X}$, coordinate representation of $X$, one has $\hat{X}\circ \nu= \nu\circ\hat{X}.$
Clearly, $\textbf{Vec}({}_\nu G_{k|l}(m|n))$ is a super Lie algebra. One can show that $\underline{\textbf{Vec}}\big({}_\nu G_{k|l}(m|n)\big)$ is a super Lie subalgebra of  $\textbf{Vec}({}_\nu G_{k|l}(m|n)).$ Set $$\mathfrak{h}:=\rho_a^{-1}\Big(\underline{\textbf{Vec}}\big({}_\nu G_{k|l}(m|n)\big)\Big).$$
Because $\rho_a$ is a super Lie algebra morphism, so $\mathfrak{h}$ is a super Lie algebra.
One can show that $\underline{\textbf{Vec}}\big({}_\nu G_{k|l}(m|n)\big)$ is
a nontrivial subalgebra of $\textbf{Vec}\big({}_\nu G_{k|l}(m|n)\big)$. Thus
 $\mathfrak{h}=\mathfrak h_0\oplus\mathfrak h_1$ is a nontrivial subalgebra
 of $\mathfrak{gl}(m|n).$ 
 
First, we show that $\underline{\textbf{Vec}}\big({}_\nu G_{k|l}(m|n)\big)$ has a nonzero element. To this end, we give an example for which $\underline{\textbf{Vec}}\big({}_\nu G_{k|l}(m|n)\big)$ is not trivial.
\begin{example}
Let $(x;e)$ be a coordinate system of the $1|1$-dimension superdomain $\mathbb R^{1|1}$
and let $\nu:C^\infty_{\mathbb R}\otimes \wedge \mathbb R\rightarrow C^\infty_{\mathbb R}\otimes \wedge \mathbb R$ be the involution as follows:
  \begin{align*}
 &\nu(1)=e,\\
  &\nu(e)=1,\\
  &\nu(fe)=f\nu(e), \qquad \forall f\in C^\infty(\mathbb R).
 \end{align*}
The supermanifold ${}_\nu G_{0|1}(1|2)$ is constructed by gluing superdomains $U_1, U_2$
and
$U_3$ with $0|1\times 1|2$- matrices 
$A_1, A_2$
and
$A_3$ as their labels respectively, where $U_i=U_{I_i|R_i}$, $i=1,2,3$. By section $2$, $I_i\subset \{1\}$ and $R_i\subset\{1, 2\}$ such that the sum of the numbers of their elements is equal to $1$. So one may has
\begin{align*}
& I_1=\varnothing,\qquad R_1=\{1\}\qquad\qquad A_1=\left[
\begin{array}{c|cc}
 e & 1 & x \\
\end{array}
\right] \\
& I_2=\varnothing,\qquad R_1=\{2\}\qquad\qquad A_2=\left[
\begin{array}{c|cc}
 e & x & 1 \\
\end{array}
\right]\\
& I_3=\{1\},\qquad R_3=\varnothing \qquad\qquad A_3=\left[
\begin{array}{c|cc}
 1\nu & \nu e & x \\
\end{array}
\right].
\end{align*}
Therefore all coordinate transformation $g_{I|R}^*$ may be defined as stated in section 2. For example it is easily seen that the coordinate transformation $g_{12}^*$ is given by
\begin{align}\label{coortrans}
g^*_{12}(x)=\dfrac{1}{x},\qquad g^*_{12}(e)=\dfrac{e}{x}.
\end{align}
By Batchelor theorem, there exists a vector bundle $\xi=(E,\pi,S^1)$ such that
$${}_\nu G_{0|1}(1|2)\cong (S^1,\Gamma(\wedge E)).$$
Let $\mathcal J$ be the sheaf of nilpotent elements of $\mathcal O$, the structure sheaf of ${}_\nu G_{0|1}(1|2)$. It is seen that the sheaf of sections of $\xi$ is isomorphic to $\mathcal O/\mathcal J^2$. Thus the equations in (\ref{coortrans}) show that $\xi$ is the mobius band considered as a line bundle over $S^1$. We have to show that there exists a right invariant vector field on
${}_\nu G_{0|1}(1|2)$ associated to action
$${}_\nu G_{0|1}(1|2)\times GL(1|2)\rightarrow {}_\nu G_{0|1}(1|2)$$
such that their coordinate representations commute with $\nu.$
For this, let $X_1$ be a right invariant vector field on Lie group $S^1.$ 
Then $X_1$ determines a right invariant vector field in 
$\textbf{Vec}({}_\nu G_{0|1}(1|2)).$
To this end, since $\mathcal O\cong C^{\infty}_{S^1}\oplus sec\xi$, it is sufficient to show that $X_1$ induces $\tilde{X}_1$ a covariant derivative on $\xi$. To proceed, we need to specify the sections of $\xi$ as follows  

Let $\mathbb P:\mathbb R\rightarrow S^1$ be the exponential map $t\mapsto exp2\pi\sqrt{-1}t$. 
Since $\mathbb{R}$ is contractible, the pull back bundle $\mathbb P^*\xi$ is a trivial line bundle. So one has $sec\mathbb P^*\xi\cong C^{\infty}(\mathbb{R})$. Therefore, one may show that there is a $1-1$ correspondence between sections of $\xi$ and $\mathfrak P$ the set of smooth periodic functions, say $f:\mathbb R\rightarrow \mathbb R$, with period $2$, such that for each $t\in \mathbb{R}$, $f(t+1)= -f(t).$

On the other hand, since $\mathbb P$ is local diffeomorphism, $X_1$ defines $Y$ a vector field on $\mathbb{R}$ as follows. Set $Y(t)=\mathbb P_*^{-1}(X_1(\mathbb P(t))$. Obviously, Y defines a derivations on $\mathfrak P$. So one gets a right invariant vector field  in $\underline{\textbf{Vec}}({}_\nu G_{0|1}(1|2))$. Thus
$$\underline{\textbf{Vec}}\big({}_\nu G_{0|1}(1|2)\big)\neq \{0\}.$$
\end{example}

\medskip
Let $\mathfrak h=\mathfrak h_0\oplus \mathfrak h_1$, then $\mathfrak h_0$ is a Lie algebra. Let $H_0$ be the unique simply connected Lie Group such that $h_0= Lie(H_0).$
\begin{theorem}
There exists a representation $\sigma: H_0\to Aut(h)$ such that $(H_0,\mathfrak{h},\sigma)$ is a super Harish-Chandra pair.
\end{theorem}
\begin{proof}
Let $ad:\mathfrak h_0\rightarrow End(\mathfrak h)$ with $X\mapsto ad_X$ be the adjoint representation of $\mathfrak h_0$ and $exp:\mathfrak h_0\rightarrow H_0$ be the exponential map. Since $End(\mathfrak h)$ is the set of even morphisms, so we have the map $Exp:End(\mathfrak h)\rightarrow Aut(\mathfrak h)$. See \cite{Rothstein} for more details. One can define
 locally $\sigma:=Exp\circ ad\circ exp^{-1}.$ By Campbell-Hausdorff theorem  $\sigma$ is a homomorphism . Since $H_0$ is connected, $\sigma$ may be defined globally on $H_0$ so it is a representation of $H_0$ over $\mathfrak h.$ It is seen that $(H_0,\mathfrak{h},\sigma)$ is a super Harish-Chandra pair.
 \end{proof}
 
 We denote by $H$ the super Lie group associate with this super Harish-Chandra pair. It is seen that $H$ is a super Lie group associated to odd involution $\nu$. So we call it $\nu$-Lie Group. 
It can be shown that $H$ acts on ${}_\nu G_{k|l}(m|n)$.

\newpage
\providecommand{\bysame}{\leavevmode\hbox to3em{\hrulefill}\thinspace}
\providecommand{\MR}{\relax\ifhmode\unskip\space\fi MR }

\providecommand{\href}[2]{#2}

\end{document}